\documentclass{amsart}
\usepackage{amsmath,amssymb}
\usepackage[dvips]{graphics}
\usepackage[all]{xy}
\newtheorem{Theorem}{Theorem}[section]
\newtheorem{Lemma}[Theorem]{Lemma}
\newtheorem{Proposition}[Theorem]{Proposition}

\theoremstyle{definition}

\newtheorem{Example}[Theorem]{Example}

\theoremstyle{remark}
\newtheorem{Remark}[Theorem]{Remark}

\def\labelenumi{\rm ({\makebox[0.8em][c]{\theenumi}})}
\def\theenumi{\arabic{enumi}}
\def\labelenumii{\rm {\makebox[0.8em][c]{(\theenumii)}}}
\def\theenumii{\roman{enumii}}

\renewcommand{\thefigure}
{\arabic{section}.\arabic{figure}}
\makeatletter
\@addtoreset{figure}{section}
\def\@thmcountersep{-}
\makeatother

\numberwithin{equation}{section}

\begin{document} 

\title{On intrinsically knotted or completely 3-linked graphs}

\author{Ryo Hanaki}
\address{Department of Mathematics, Nara University of Education, Takabatake, Nara 630-8305, Japan}
\email{hanaki@nara-edu.ac.jp}

\author{Ryo Nikkuni}
\address{Department of Mathematics, School of Arts and Sciences, Tokyo Woman's Christian University, 2-6-1 Zempukuji, Suginami-ku, Tokyo 167-8585, Japan}
\email{nick@lab.twcu.ac.jp}
\thanks{The second author was partially supported by Grant-in-Aid for Young Scientists (B) (No. 21740046), Japan Society for the Promotion of Science.}

\author{Kouki Taniyama}
\address{Department of Mathematics, School of Education, Waseda University, Nishi-Waseda 1-6-1, Shinjuku-ku, Tokyo, 169-8050, Japan}
\email{taniyama@waseda.jp}
\thanks{The third author was partially supported by Grant-in-Aid for Scientific Research (C) (No. 21540099), Japan Society for the Promotion of Science.}

\author{Akiko Yamazaki}
\address{Division of Mathematics, Graduate School of Science, Tokyo Woman's Christian University, 2-6-1 Zempukuji, Suginami-ku, Tokyo 167-8585, Japan}
\email{smilebimoch@khc.biglobe.ne.jp}

\subjclass{Primary 57M15; Secondary 57M25}

\date{}


\keywords{Spatial graph, Intrinsic knottedness, $\triangle Y$-exchange, $Y \triangle$-exchange}

\begin{abstract}
We say that a graph is intrinsically knotted or completely $3$-linked if every embedding of the graph into the $3$-sphere contains a nontrivial knot or a $3$-component link any of whose $2$-component sublink is nonsplittable. We show that a graph obtained from the complete graph on seven vertices by a finite sequence of $\triangle Y$-exchanges and $Y \triangle$-exchanges is a minor-minimal intrinsically knotted or completely 3-linked graph. 
\end{abstract}

\maketitle

\section{Introduction} 

Throughout this paper we work in the piecewise linear category. Let $f$ be an embedding of a finite graph $G$ into the $3$-sphere. Then $f$ is called a {\it spatial embedding} of $G$ and $f(G)$ is called a {\it spatial graph}. We denote the set of all spatial embeddings of $G$ by ${\rm SE}(G)$. We call a subgraph $\gamma$ of $G$ which is homeomorphic to the circle a {\it cycle} of $G$. For a positive integer $n$, $\Gamma^{(n)}(G)$ denotes the set of all cycles of $G$ if $n=1$ and the set of all unions of mutually disjoint $n$ cycles of $G$ if $n\ge 2$. In particular, we denote $\Gamma^{(1)}(G)$ by $\Gamma(G)$ simply. For an element $\lambda$ in $\Gamma^{(n)}(G)$ and a spatial embedding $f$ of $G$, $f(\lambda)$ is none other than a knot if $n=1$ and an $n$-component link if $n\ge 2$. 

A graph $G$ is said to be {\it intrinsically linked} (IL) if for every spatial embedding $f$ of $G$, $f(G)$ contains a nonsplittable $2$-component link. Conway-Gordon \cite{CG83} and Sachs \cite{S84} showed that $K_{6}$ is IL, where $K_{m}$ denotes the {\it complete graph} on $m$ vertices. Moreover, IL graphs have been completely characterized as follows. For a graph $G$ and an edge $e$ of $G$, we denote the subgraph $G\setminus {\rm int}e$ by $G-e$. Let $e=\overline{uv}$ is an edge of $G$ which is not a loop. We call the graph which is obtained from $G-e$ by identifying the end vertices $u$ and $v$ the {\it edge contraction of $G$ along $e$} and denote it by $G/e$. A graph $H$ is called a {\it minor} of a graph $G$ if there exists a subgraph $G'$ of $G$ and the edges $e_{1},e_{2},\ldots,e_{m}$ of $G'$ such that $H$ is obtained from $G'$ by a sequence of edge contractions along $e_{1},e_{2},\ldots,e_{m}$. A minor $H$ of $G$ is called a {\it proper minor} if $H$ does not equal $G$. Let ${\mathcal P}$ be a property for graphs which is {\it closed} under minor reductions; that is, for any graph $G$ which does not have ${\mathcal P}$, all minors of $G$ also do not have ${\mathcal P}$. A graph $G$ is said to be {\it minor-minimal} with respect to ${\mathcal P}$ if $G$ has ${\mathcal P}$ but all proper minors of $G$ do not have ${\mathcal P}$. Note that $G$ has ${\mathcal P}$ if and only if $G$ has a minor-minimal graph with respect to ${\mathcal P}$ as a minor. By the famous theorem of Robertson-Seymour \cite{RS04}, there are finitely many minor-minimal graphs with respect to ${\mathcal P}$. Ne\v{s}et\v{r}il-Thomas \cite{NT85} showed that IL is closed under minor reductions, and Robertson-Seymour-Thomas \cite{RST95} showed that the set of all minor-minimal graphs with respect to IL equals the {\it Petersen family} which is the set of all graphs obtained from $K_{6}$ by a finite sequence of {\it $\triangle Y$-exchanges} and {\it $Y \triangle$-exchanges}. Here a $\triangle Y$-exchange is an operation to obtain a new graph $G_{Y}$ from a graph $G_{\triangle}$ by removing all edges of a cycle $\triangle$ of $G_{\triangle}$ with exactly three edges $\overline{uv},\overline{vw}$ and $\overline{wu}$, and adding a new vertex $x$ and connecting it to each of the vertices $u,v$ and $w$ as illustrated in Fig. \ref{Delta-Y} (we often denote $\overline{ux}\cup \overline{vx}\cup \overline{wx}$ by $Y$). A $Y \triangle$-exchange is the reverse of this operation. This family contains exactly seven graphs as illustrated in Fig. \ref{Petersen}, where $G\to G'$ means that $G'$ can be obtained from $G$ by a single $\triangle Y$-exchange. Note that $P_{10}$ is isomorphic to the {\it Petersen graph}. We remark here that if $G_{\triangle}$ is IL then $G_{Y}$ is also IL \cite{MRS88}, and if $G_{Y}$ is IL then $G_{\triangle}$ is also IL \cite{RST95}. Namely $\triangle Y$ and $Y \triangle$-exchanges preserve IL. 

\begin{figure}[htbp]
      \begin{center}
\scalebox{0.45}{\includegraphics*{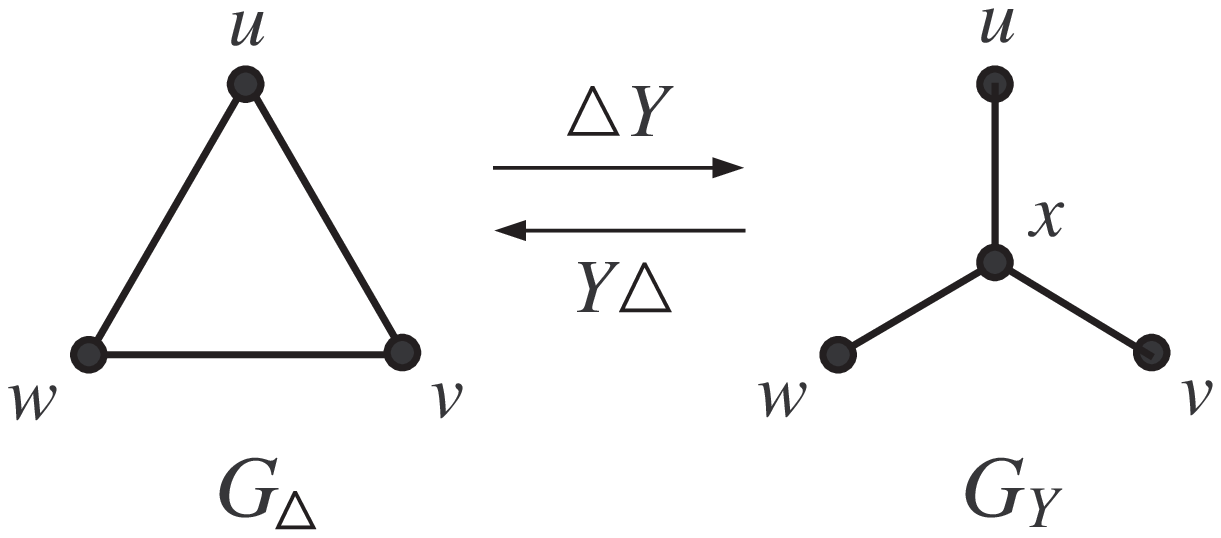}}
      \end{center}
   \caption{}
  \label{Delta-Y}
\end{figure} 
\begin{figure}[htbp]
      \begin{center}
\scalebox{0.45}{\includegraphics*{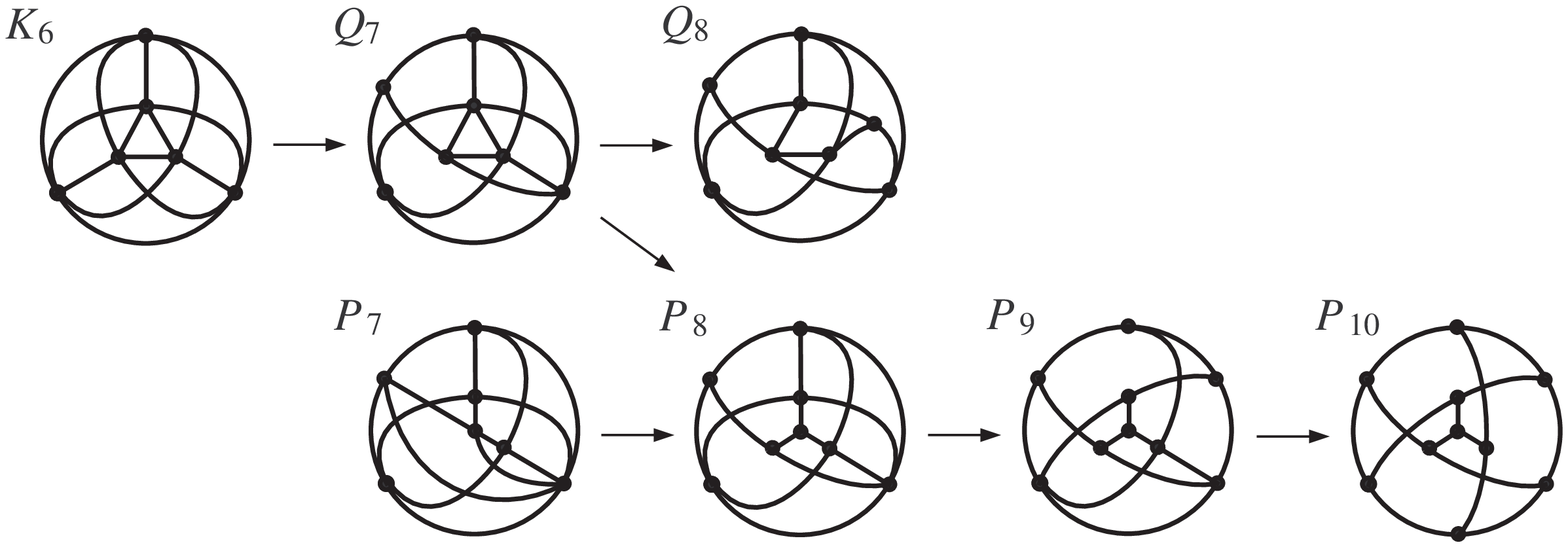}}
      \end{center}
   \caption{}
  \label{Petersen}
\end{figure} 

On the other hand, a graph $G$ is said to be {\it intrinsically knotted} (IK) if for every spatial embedding $f$ of $G$, $f(G)$ contains a nontrivial knot. Conway-Gordon \cite{CG83} showed that $K_{7}$ is IK. Fellows and Langston \cite{FL88} showed that IK is closed under minor reductions, and Motwani-Raghunathan-Saran \cite{MRS88} showed that $K_{7}$ is a minor-minimal IK graph. Although additional minor-minimal IK graphs are known by Kohara-Suzuki \cite{KS92} and Foisy \cite{F02}, \cite{F04}, IK graphs have not been completely characterized yet. We remark here that if $G_{\triangle}$ is IK then $G_{Y}$ is also IK \cite{MRS88}, but if $G_{Y}$ is IK then $G_{\triangle}$ may not always be IK. Namely the $Y \triangle$-exchange does not preserve IK in general. Actually Flapan-Naimi \cite{FN08} exhibited that there exists a graph $G_{FN}$ which is obtained from $K_{7}$ by five times of $\triangle Y$-exchanges and twice $Y \triangle$-exchanges such that it is not IK. We call the set of all graphs obtained from $K_{7}$ by a finite sequence of $\triangle Y$ and $Y \triangle$-exchanges the {\it Heawood family}.\footnote{In \cite{H06}, van der Holst call the set of all graphs obtained from $K_{7}$ or $K_{3,3,1,1}$ by a finite sequence of $\triangle Y$ and $Y \triangle$-exchanges the Heawood family, where $K_{3,3,1,1}$ is the complete $4$-partite graph on $3+3+1+1$ vertices.} This family contains exactly twenty graphs as illustrated in Fig. \ref{Heawood}, where $G\to G'$ means that $G'$ can be obtained from $G$ by a single $\triangle Y$-exchange. Note that $C_{14}$ is isomorphic to the {\it Heawood graph}, see Remark \ref{Hea}. 

\begin{figure}[htbp]
      \begin{center}
\scalebox{0.435}{\includegraphics*{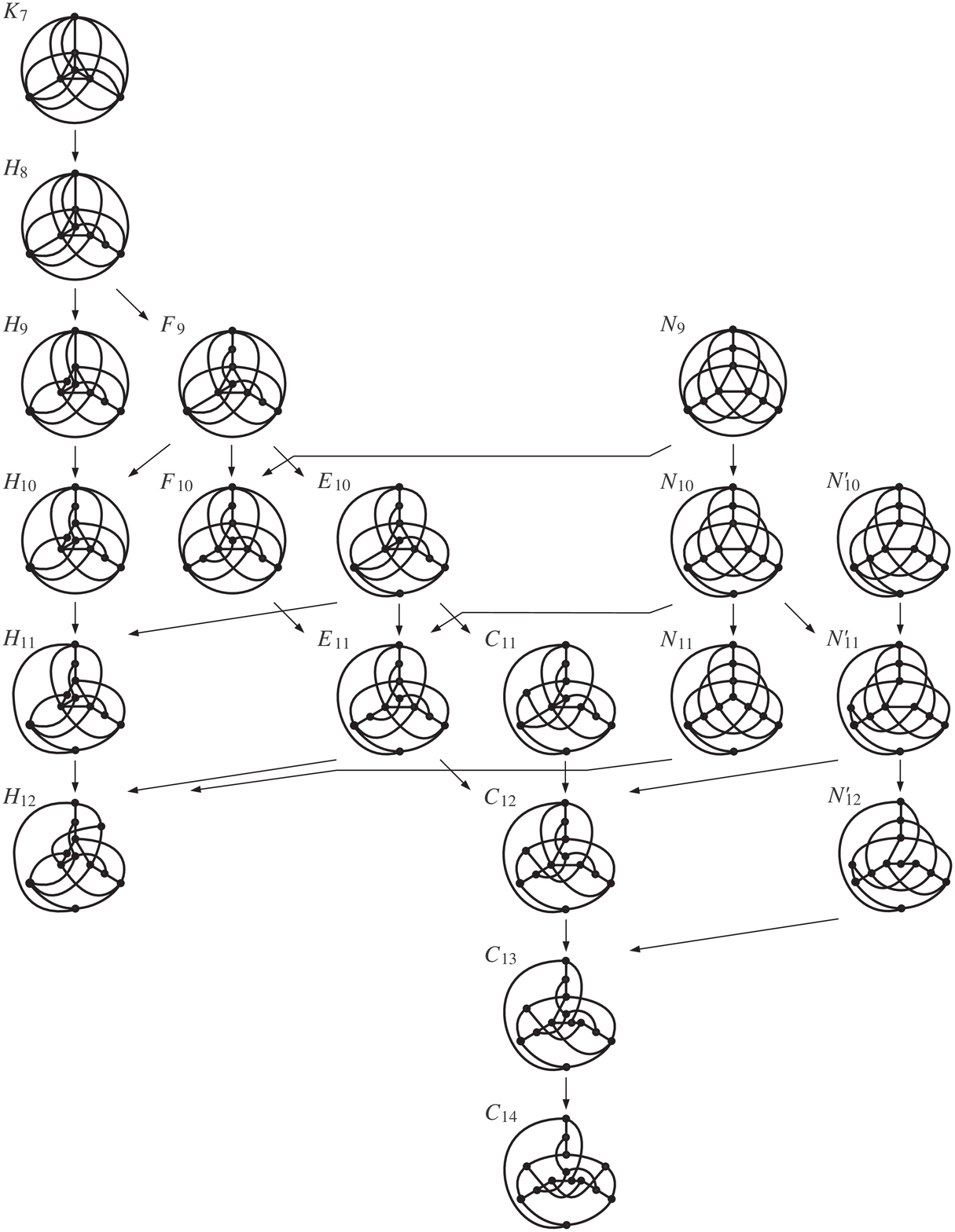}}
      \end{center}
   \caption{}
  \label{Heawood}
\end{figure} 

Kohara-Suzuki \cite{KS92} showed that a graph $G$ in the Heawood family is a minor-minimal IK graph if $G$ is obtained from $K_{7}$ by a finite sequence of $\triangle Y$-exchanges, namely $G$ is one of fourteen graphs $K_{7}$, $H_{8},H_{9},\ldots,H_{12}$, $F_{9},F_{10}$, $E_{10},E_{11}$ and $C_{11},C_{12},\ldots,C_{14}$.\footnote{We remark that one edge of $F_{10}$ in \cite[Fig. 5]{KS92} is wanting.} On the other hand, $N'_{10}$ is isomorphic to $G_{FN}$, namely $N'_{10}$ is not IK. Our first purpose in this paper is to determine completely when a graph in the Heawood family is IK as follows. 

\begin{Theorem}\label{main1} 
Let $G$ be a graph in the Heawood family. Then the following are equivalent: 
\begin{enumerate}
\item $G$ is IK, 
\item $G$ is obtained from $K_{7}$ by a finite sequence of $\triangle Y$-exchanges,
\item $\Gamma^{(3)}(G)$ is the empty set. 
\end{enumerate}
\end{Theorem}

Namely, each of the graphs $N_{9},N_{10},N_{11},N'_{10},N'_{11}$ and $N'_{12}$ in the Heawood family is not IK, and only these graphs in the Heawood family contain a union of mutually disjoint three cycles. Our second purpose in this paper is to show that any of the graphs in the Heawood family is a minor-minimal graph with respect to a certain kind of intrinsic nontriviality even if it is not IK. We say that a graph $G$ is {\it intrinsically knotted or completely $3$-linked} (I(K or C3L)) if for every spatial embedding $f$ of $G$, $f(G)$ contains a nontrivial knot or a $3$-component link any of whose $2$-component sublink is nonsplittable. Note that an IK graph is I(K or C3L). As we will show in Proposition \ref{closed}, I(K or C3L) is closed under minor reductions. Then we have the the following. 

\begin{Theorem}\label{main2} 
All of the graphs in the Heawood family are minor-minimal I(K or C3L) graphs. 
\end{Theorem}

Actually, each of the graphs $N_{9},N_{10},N_{11},N'_{10},N'_{11}$ and $N'_{12}$ in the Heawood family is not IK but I(K or C3L), and they are minor-minimal with respect to I(K or C3L). 

\begin{Remark}\label{rem} 
\begin{enumerate}
\item A graph $G$ is said to be {\it intrinsically $n$-linked} (I$n$L) if for every spatial embedding $f$ of $G$, $f(G)$ contains a nonsplittable $n$-component link \cite{FNP01} \cite{FPFN01}. Note that I$2$L coincides with IL. Let $G$ be a graph in the Heawood family which is not IK. Then we will show in Example \ref{not3L} that there exists a spatial embedding $f$ of $G$ such that $f(G)$ does not contain a nonsplittable $3$-component link. Namely $G$ is neither IK nor I$3$L. 
\item A graph $G$ is said to be {\it intrinsically knotted or $3$-linked} (I(K or 3L)) if for every spatial embedding $f$ of $G$, $f(G)$ contains a nontrivial knot or a nonsplittable $3$-component link \cite{F06}. It is clear that I(K or C3L) implies I(K or 3L), but the converse is not true. Actually in \cite{F06}, although Foisy discovered an I(K or 3L) graph $G$ and exhibit a spatial embedding $f$ of $G$ such that $f(G)$ contains a nonsplittable $3$-component link but does not contain a nontrivial knot, each of the nonsplittable $3$-component links in $f(G)$ contains a split $2$-component sublink. 
\end{enumerate}
\end{Remark}

The rest of this paper is organized as follows. In the next section, we show the general results about graph minors, $\triangle Y$-exchanges and spatial graphs. We prove Theorems \ref{main1} and \ref{main2} in sections $3$ and $4$, respectively. 

\section{Graph minors, $\triangle Y$-exchanges and spatial graphs} 

Let $H$ be a minor of a graph $G$. Then there exists a natural injection 
\begin{eqnarray*}
\Psi^{(n)}=\Psi_{H,G}^{(n)}:\Gamma^{(n)}(H)\longrightarrow \Gamma^{(n)}(G)
\end{eqnarray*}
for any positive integer $n$. In particular, we denote $\Psi^{(1)}$ by $\Psi$ simply. Let $f$ be a spatial embedding of $G$ and $e$ an edge of $G$ which is not a loop. Then by contracting $f(e)$ into one point, we obtain a spatial embedding $\psi(f)$ of $G/e$. Similarly we also can obtain a spatial embedding $\psi(f)$ of $H$ from $f$. Thus we obtain a map 
\begin{eqnarray*}
\psi=\psi_{G,H}:{\rm SE}(G)\longrightarrow {\rm SE}(H). 
\end{eqnarray*}
Then we immediately have the following. 

\begin{Proposition}\label{map2} 
For a spatial embedding $f$ of $G$ and an element $\lambda$ in $\Gamma^{(n)}(H)$, $\psi(f)(\lambda)$ is ambient isotopic to $f\left(\Psi^{(n)}(\lambda)\right)$. \hfill $\square$
\end{Proposition}

Now we show the following. 

\begin{Proposition}\label{closed} 
I(K or C3L) is closed under minor reductions. 
\end{Proposition}

\begin{proof}
Let $G$ be a graph which is not I(K or C3L) and $H$ a minor of $G$. Let $f$ be a spatial embedding of $G$ which contains neither a nontrivial knot nor a $3$-component link any of whose $2$-component sublink is nonsplittable. Then by Proposition \ref{map2}, $\psi(f)$ also contains neither a nontrivial knot nor a $3$-component link any of whose $2$-component sublink is nonsplittable. This implies that $H$ is not I(K or C3L). 
\end{proof}

\begin{Remark}\label{closed2} 
Proposition \ref{map2} also implies that IK, I$n$L and I(K or 3L) are closed under minor reductions. 
\end{Remark}

Let $G_{\triangle}$ and $G_{Y}$ be two graphs such that $G_{Y}$ is obtained from $G_{\triangle}$ by a single $\triangle Y$-exchange as we said in the previous section. Let $\lambda$ be an element in $\Gamma^{(n)}(G_{\triangle})$ which does not contain $\triangle$. Then there exists an element $\Phi^{(n)}(\lambda)$ in $\Gamma^{(n)}(G_{Y})$ such that $\lambda\setminus \triangle=\Phi^{(n)}(\lambda)\setminus Y$. Thus we obtain a map 
\begin{eqnarray*}
\Phi^{(n)}=\Phi_{G_{\triangle},G_{Y}}^{(n)}:\left\{\lambda\in\Gamma^{(n)}(G_{\triangle})\ |\ \lambda\not\supset \triangle\right\}\longrightarrow \Gamma^{(n)}(G_{Y})
\end{eqnarray*}
for any positive integer $n$. In particular, we denote $\Phi^{(1)}$ by $\Phi$ simply. Note that $\Phi^{(n)}$ is surjective and the inverse image of $\lambda$ by $\Phi^{(n)}$ contains at most two elements in $\Gamma^{(n)}(G_{\triangle})$ for any element $\lambda$ in $\Gamma^{(n)}(G_{Y})$. Note also that the surjectivity of $\Phi^{(n)}$ implies the following. 

\begin{Proposition}\label{main1lem2} 
For $n\ge 2$, if $\Gamma^{(n)}(G_{\triangle})$ is the empty set, then $\Gamma^{(n)}(G_{Y})$ is also the empty set. \hfill $\square$
\end{Proposition}

Let $f$ be a spatial embedding of $G_{Y}$ and $D$ a $2$-disk in the $3$-sphere such that $D\cap f(G_{Y})=f(Y)$ and $\partial D \cap f(G_{Y}) = \{f(u),f(v),f(w)\}$. Let $\varphi(f)$ a spatial embedding of $G_{\triangle}$ such that $\varphi(f)(x)=f(x)$ for $x\in G_{Y}\setminus Y$ and $\varphi(f)(G_{\triangle})=\left(f(G_{Y})\setminus f(Y)\right)\cup \partial D$. Thus we obtain a map 
\begin{eqnarray*}
\varphi=\varphi_{G_{Y},G_{\triangle}}:{\rm SE}(G_{Y})\longrightarrow {\rm SE}(G_{\triangle}). 
\end{eqnarray*}
Then we immediately have the following. 

\begin{Proposition}\label{map} 
For a spatial embedding $f$ of $G_{Y}$ and an element $\lambda$ in $\Gamma^{(n)}(G_{Y})$, $f(\lambda)$ is ambient isotopic to $\varphi(f)(\lambda')$ for each element $\lambda'$ in the inverse image of $\lambda$ by $\Phi^{(n)}$. \hfill $\square$
\end{Proposition}

Now we show the following lemmas. 

\begin{Lemma}\label{DY} 
If $G_{\triangle}$ is I(K or C3L), then $G_{Y}$ is also I(K or C3L). 
\end{Lemma}

\begin{proof}
Assume that $G_{Y}$ is not I(K or C3L), namely there exists a spatial embedding $f$ of $G_{Y}$ which contains neither a nontrivial knot nor a $3$-component link any of whose $2$-component sublink is nonsplittable. In the following we show that $\varphi(f)(G_{\triangle})$ also contains neither a nontrivial knot nor a $3$-component link any of whose $2$-component sublink is nonsplittable. Let $\gamma$ be an element in $\Gamma(G_{\triangle})$. If $\gamma$ is not $\triangle$, then $\varphi(f)(\gamma)$ is ambient isotopic to $f(\Phi(\gamma))$ by Proposition \ref{map} and $f(\Phi(\gamma))$ is a trivial knot by the assumption. Since $\varphi(f)(\triangle)$ is also a trivial knot, it follows that $\varphi(f)(G_{\triangle})$ does not contain a nontrivial knot. Let $\lambda$ be an element in $\Gamma^{(3)}(G_{\triangle})$. If $\lambda$ does not contain $\triangle$, then $\varphi(f)(\lambda)$ is ambient isotopic to $f(\Phi^{(3)}(\lambda))$ by Proposition \ref{map} and $f(\Phi^{(3)}(\lambda))$ is a $3$-component link which contains a split $2$-component sublink by the assumption. If $\lambda$ contains $\triangle$, then $\varphi(f)(\lambda)$ is a split $3$-component link. Thus we see that $\varphi(f)(G_{\triangle})$ does not contain a $3$-component link any of whose $2$-component sublink is nonsplittable. 
\end{proof}

\begin{Lemma}\label{minor-min-lem} 
If $G_{Y}$ is minor-minimal for I(K or C3L), then $G_{\triangle}$ is also minor-minimal for I(K or C3L). 
\end{Lemma}

\begin{proof}
In the following we show that for any edge $e$ of $G_{\triangle}$ which is not a loop, there exists a spatial embedding $f$ of $G_{\triangle}-e$ and a spatial embedding $g$ of $G_{\triangle}/e$ such that each of $f(G_{\triangle}-e)$ and $g(G_{\triangle}/e)$ contains neither a nontrivial knot nor a $3$-component link any of whose $2$-component sublink is nonsplittable. If $e$ is not $\overline{uv}$, $\overline{vw}$ or $\overline{wu}$, then there exists a spatial embedding $f'$ of $G_{Y}-e$ and a spatial embedding $g'$ of $G_{Y}/e$ such that both $f'(G_{Y}-e)$ and $g'(G_{Y}/e)$ contains neither a nontrivial knot nor a $3$-component link any of whose $2$-component sublink is nonsplittable. Note that $G_{Y}-e$ (resp. $G_{Y}/e$) is obtained from $G_{\triangle}-e$ (resp. $G_{\triangle}/e$) by a single $\triangle Y$-exchange at the same $\triangle$. Then we see that each of $\varphi(f')(G_{\triangle}-e)$ and $\varphi(g')(G_{\triangle}/e)$ contains neither a nontrivial knot nor a $3$-component link any of whose $2$-component sublink is nonsplittable in the similar way as the proof of Lemma \ref{DY}. If $e$ is one of $\overline{uv}$, $\overline{vw}$ and $\overline{wu}$, we may assume that $e=\overline{uv}$ without loss of generality. Now there exists a spatial embedding $f'$ of $G_{Y}/\overline{xw}$ such that $f'(G_{Y}/\overline{xw})$ contains neither a nontrivial knot nor a $3$-component link any of whose $2$-component sublink is nonsplittable. Then we can see that $G_{\triangle}-\overline{uv}=G_{Y}/\overline{xw}$. On the other hand, there exists a spatial embedding $g'$ of $G_{Y}/\overline{xv}/\overline{xu}$ such that $g'(G_{Y}/\overline{xv}/\overline{xu})$ contains neither a nontrivial knot nor a $3$-component link any of whose $2$-component sublink is nonsplittable. Take a $2$-disk $D'$ in the $3$-sphere such that $D'\cap g'(G_{Y}/\overline{xv}/\overline{xu})=g'(\overline{uw})$ and $\partial D' \cap g'(G_{Y}/\overline{xv}/\overline{xu}) = \{g'(u),g'(w)\}$. Then $\left(g'(G_{Y}/\overline{xv}/\overline{xu})\setminus {\rm int}g'(\overline{uw})\right)\cup \partial D'$ may be regarded as the image of a spatial embedding of $G_{\triangle}/\overline{uv}$, which is denoted by $g$. It is clear that $g(G_{\triangle}/\overline{uv})$ contains neither a nontrivial knot nor a $3$-component link any of whose $2$-component sublink is nonsplittable. 
\end{proof}

\begin{Remark}\label{ot_rem} 
{\rm 
Lemma \ref{minor-min-lem} has already been proven by Ozawa-Tsutsumi in a more general form \cite[Lemma 3.1, Exercise 3.2]{OT07}. But we give a proof of Lemma \ref{minor-min-lem} as described above for the reader's convenience. 
}
\end{Remark}

\section{Proof of Theorem \ref{main1}} 

\begin{Lemma}\label{main1lem} 
Each of the graphs $N_{9},N_{10},N_{11},N'_{10},N'_{11}$ and $N'_{12}$ in the Heawood family is not IK. 
\end{Lemma}

\begin{proof}
Since the case of $N'_{10}$ has been already shown by Flapan-Naimi \cite{FN08}, we show that each of the graphs $N_{9},N_{10},N_{11},N'_{11}$ and $N'_{12}$ is not IK. Let $f_{9}$ be the spatial embedding of $N_{9}$ as illustrated in Fig. \ref{no_knots}. Then it can be checked directly that $f_{9}(N_{9})$ does not contain a nontrivial knot. Thus $N_{9}$ is not IK. Let $f_{10}$ be the spatial embedding of $N_{10}$ as illustrated in Fig. \ref{no_knots}. Let $\varphi_{N_{10},N_{9}}$ be the map from ${\rm SE}(N_{10})$ to ${\rm SE}(N_{9})$ induced by the $Y \triangle$-exchange from $N_{10}$ to $N_{9}$ at the marked $Y$ as illustrated in Fig. \ref{no_knots}. Then it is clear that $\varphi(f_{10})=f_{9}$. Since $f_{9}(N_{9})$ does not contain a nontrivial knot, by Proposition \ref{map} it follows that $f_{10}(N_{10})$ also does not contain a nontrivial knot. Namely $N_{10}$ is not IK. By repeating this argument, we can see that each of the graphs $N_{11}, N'_{11}$ and $N'_{12}$ is also not IK, see Fig. \ref{no_knots}. 
\end{proof}

\begin{figure}[htbp]
      \begin{center}
\scalebox{0.45}{\includegraphics*{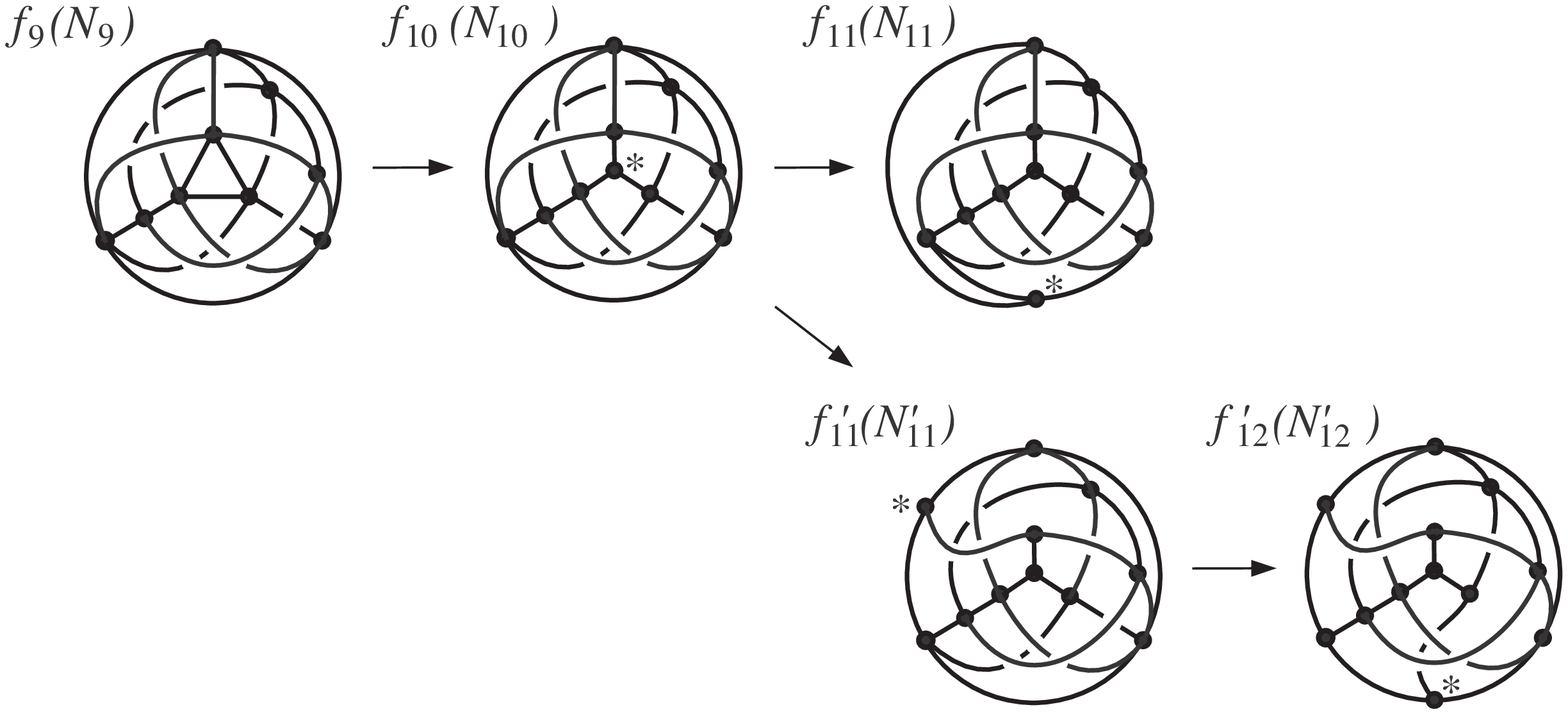}}
      \end{center}
   \caption{}
  \label{no_knots}
\end{figure} 

\begin{proof}[Proof of Theorem \ref{main1}.] 
First we show that (1) and (2) are equivalent. Since we have already known that (2) implies (1), we show that (1) implies (2). If $G$ is IK, then by Lemma \ref{main1lem} we see that $G$ is not one of $N_{9},N_{10},N_{11},N'_{10},N'_{11}$ or $N'_{12}$. Namely $G$ is obtained from $K_{7}$ by a finite sequence of $\triangle Y$-exchanges. Next we show that (2) and (3) are equivalent. Assume that $G$ is obtained from $K_{7}$ by a finite sequence of $\triangle Y$-exchanges. Note that $\Gamma^{(3)}(K_{7})$ is the empty set. Thus by Proposition \ref{main1lem2} we see that $\Gamma^{(3)}(G)$ is the empty set. Conversely, if $G$ is one of $N_{9},N_{10},N_{11},N'_{10},N'_{11}$, and $N'_{12}$, then $\Gamma^{(3)}(G)$ is not the empty set. This completes the proof. 
\end{proof}

\begin{Remark}\label{FN_revisited} 
{\rm 
Let $f'_{11}$ be the spatial embedding of $N'_{11}$ as illustrated in Fig. \ref{no_knots} and $f'_{10}$ the spatial embedding of $N'_{10}$ as illustrated in Fig. \ref{FN}. Let $\varphi_{N'_{11},N'_{10}}$ be the map from ${\rm SE}(N'_{11})$ to ${\rm SE}(N'_{10})$ induced by the $Y \triangle$-exchange from $N'_{11}$ to $N'_{10}$ at the double-marked $Y$ as illustrated in Fig. \ref{FN}. Then it is clear that $\varphi(f'_{11})=f'_{10}$. Moreover we can see that $f'_{10}$ coincides with Flapan-Naimi's example of a spatial embedding of $N'_{10}$ whose image does not contain a nontrivial knot as illustrated in the left side of Fig. \ref{FN} \cite{FN08}. 
}
\end{Remark}

\begin{figure}[htbp]
      \begin{center}
\scalebox{0.45}{\includegraphics*{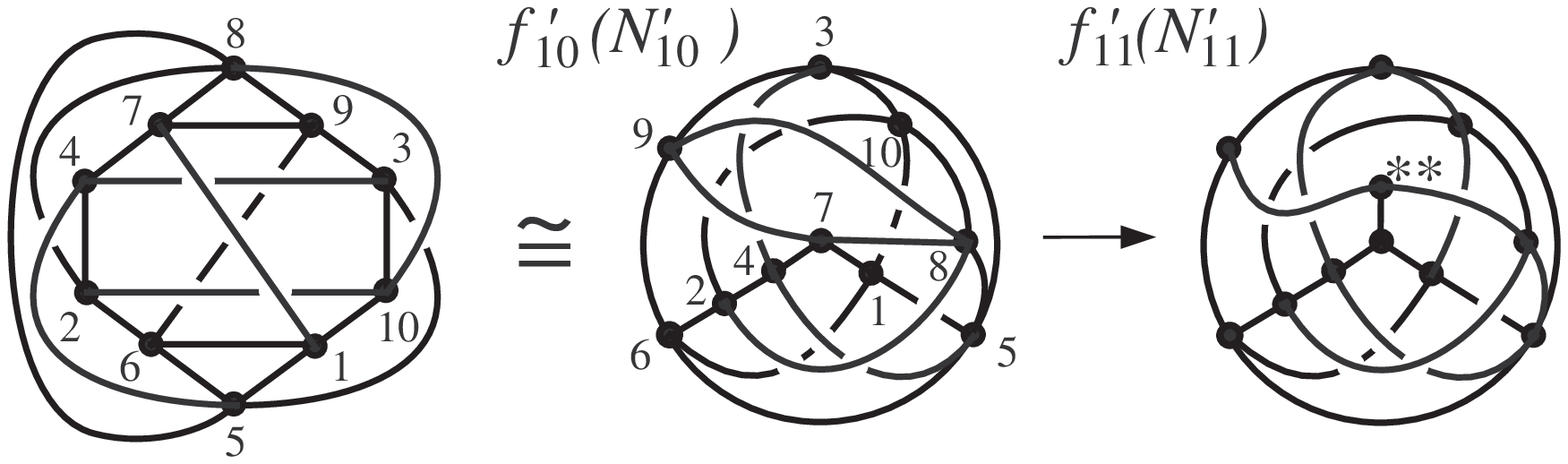}}
      \end{center}
   \caption{}
  \label{FN}
\end{figure} 
\section{Proof of Theorem \ref{main2}} 

We need some lemmas which are needed to prove Theorem \ref{main2}. 

\begin{Lemma}\label{CGTY} {\rm (Conway-Gordon \cite{CG83}, Taniyama-Yasuhara \cite{TY01})} Let $G$ be a graph in the Petersen family and $f$ a spatial embedding of $G$. Then there exists an element $\lambda$ in $\Gamma^{(2)}(G)$ such that ${\rm lk}(f(\lambda))$ is odd, where ${\rm lk}$ denotes the {\it linking number} in the $3$-sphere. 
\end{Lemma}

Let $D_{4}$ be the graph as illustrated in Fig. \ref{D_4}. We denote the set of all cycles with exactly four edges of $D_{4}$ by $\Gamma_{4}(D_{4})$. For a spatial embedding $f$ of $D_{4}$, we define 
\begin{eqnarray*}
{\alpha}(f)\equiv \sum_{\gamma\in \Gamma_{4}(D_{4})}a_{2}(f(\gamma)) \pmod{2},
\end{eqnarray*}
where $a_{2}$ denotes the second coefficient of the {\it Conway polynomial}. Note that $a_{2}(K)$ of a knot $K$ is congruent to the {\it Arf invariant} modulo two \cite{Ka83}. Then the following is known. 

\begin{Lemma}\label{TYa2} {\rm (Taniyama-Yasuhara \cite{TY01})} Let $f$ be a spatial embedding of $D_{4}$ and $\lambda,\lambda'$ all elements in $\Gamma^{(2)}(D_{4})$. If both ${\rm lk}(f(\lambda))$ and ${\rm lk}(f(\lambda'))$ are odd, then $\alpha(f)=1$. 
\end{Lemma}

\begin{figure}[htbp]
      \begin{center}
\scalebox{0.45}{\includegraphics*{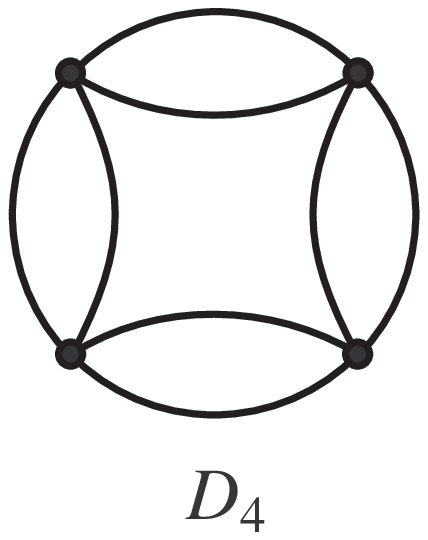}}
      \end{center}
   \caption{}
  \label{D_4}
\end{figure} 

Let $G$ be a graph which contains $D_{4}$ as a minor and $f$ a spatial embedding of $G$. Then we define
\begin{eqnarray*}
{\alpha}(f)\equiv \sum_{\gamma\in \Gamma_{4}(D_{4})}a_{2}(f(\Psi_{D_{4},G}(\gamma))) \pmod{2}.
\end{eqnarray*}

\begin{Lemma}\label{D4type} 
Let $G$ be a graph which contains $D_{4}$ as a minor and $f$ a spatial embedding of $G$. For two elements $\mu$ and $\mu'$ in $\Psi_{D_{4},G}^{(2)}\left(\Gamma^{(2)}(D_{4})\right)$, if both ${\rm lk}(f(\mu))$ and ${\rm lk}(f(\mu'))$ are odd, then $\alpha(f)=1$. 
\end{Lemma}

\begin{proof}
For two elements $\lambda$ and $\lambda'$ in $\Gamma^{(2)}(D_{4})$, we see that both ${\rm lk}(f(\Psi_{D_{4},G}^{(2)}(\lambda)))$ and ${\rm lk}(f(\Psi_{D_{4},G}^{(2)}(\lambda')))$ are odd by the assumption. Then by Proposition \ref{map2}, it follows that ${\rm lk}(\psi_{G,D_{4}}(f)(\lambda))$ and ${\rm lk}(\psi_{G,D_{4}}(f)(\lambda'))$ are also odd. Thus by Lemma \ref{TYa2}, we have that 
\begin{eqnarray*}
\alpha(f) 
& \equiv & \sum_{\gamma\in \Gamma_{4}(D_{4})}a_{2}(f(\Psi_{D_{4},G}(\gamma)))\\
& = & \sum_{\gamma\in \Gamma_{4}(D_{4})}a_{2}(\psi_{G,D_{4}}(f)(\gamma))\\
& \equiv & 1 \pmod{2}. 
\end{eqnarray*} 
\end{proof}

Now we show the following theorem, which is the most important part in the proof of Theorem \ref{main2}. 

\begin{Theorem}\label{N9FN} 
Let $G$ be $N_{9}$ or $N'_{10}$. For every spatial embedding $f$ of $G$, there exists an element $\gamma$ in $\Gamma(G)$ such that $a_{2}(f(\gamma))$ is odd, or there exists an element $\lambda$ in $\Gamma^{(3)}(G)$ such that each $2$-component sublink of $f(\lambda)$ has an odd linking number. 
\end{Theorem}

\begin{proof}
We give a label to each vertex of $G$ as illustrated in Fig. \ref{N9N10}. In the following we denote a $k$-cycle $\overline{i_{1}i_{2}}\cup \overline{i_{2}i_{3}}\cup \cdots \cup \overline{i_{k-1}i_{k}}\cup \overline{i_{k}i_{1}}$ of $G$ by $[i_{1}i_{2}\cdots i_{k}]$. 

\begin{figure}[htbp]
      \begin{center}
\scalebox{0.45}{\includegraphics*{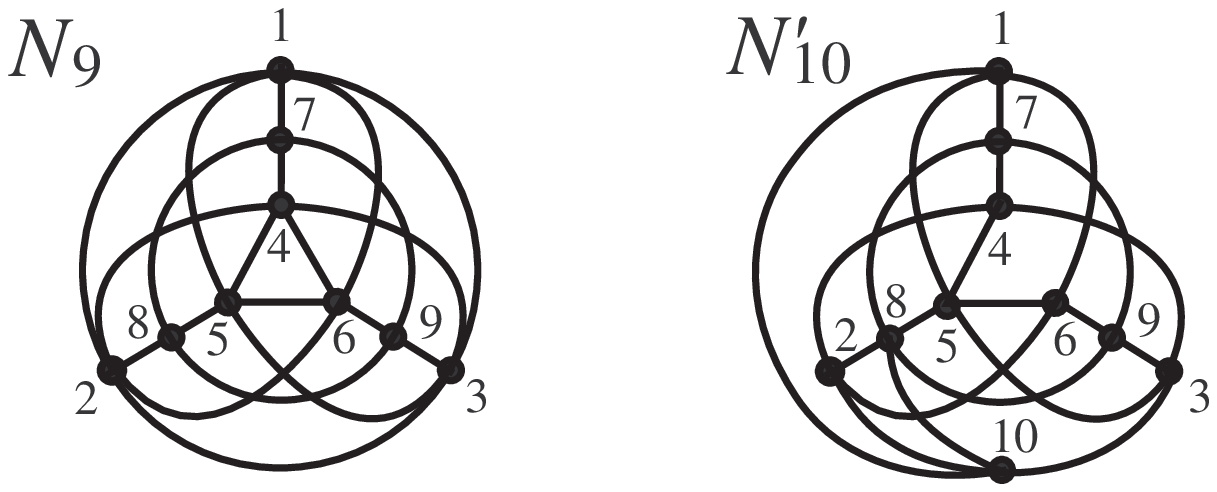}}
      \end{center}
   \caption{}
  \label{N9N10}
\end{figure} 

First we show in the case of $G=N_{9}$. Let $f$ be a spatial embedding of $N_{9}$. Note that $N_{9}$ contains $K_{6}$ as the proper minor 
\begin{eqnarray*}
\left(\left(\left(N_{9}-\overline{7\ 8}\right)-\overline{8\ 9}\right)-\overline{9\ 7}\right)/\overline{4\ 7}/\overline{5\ 8}/\overline{6\ 9}.
\end{eqnarray*}
Thus by Lemma \ref{CGTY}, there exists an element $\nu$ in $\Gamma^{(2)}(K_{6})$ such that ${\rm lk}\left(\psi_{N_{9},K_{6}}(f)(\nu)\right)$ is odd. Hence by Proposition \ref{map2}, there exists an element $\mu$ in $\Psi_{K_{6},N_{9}}^{(2)}\left(\Gamma^{(2)}(K_{6})\right)$ such that ${\rm lk}(f(\mu))$ is odd. Note that $\Psi_{K_{6},N_{9}}^{(2)}\left(\Gamma^{(2)}(K_{6})\right)$ consists of ten elements, and by the symmetry of $N_{9}$, we may assume that $\mu=[1\ 7\ 4\ 3]\cup [2\ 6\ 5\ 8]$ or $[1\ 2\ 3]\cup [4\ 5\ 6]$ without loss of generality. 

\vspace{0.2cm}
\noindent
{\bf Case 1.} $\mu=[1\ 7\ 4\ 3]\cup [2\ 6\ 5\ 8]$. 

\vspace{0.2cm}
Note that $N_{9}$ contains $P_{7}$ as the proper minor 
\begin{eqnarray*}
\left(\left(\left(\left(\left(N_{9}-\overline{6\ 1}\right)-\overline{6\ 2}\right)-\overline{6\ 4}\right)-\overline{6\ 5}\right)-\overline{6\ 9}\right)/\overline{3\ 9}. 
\end{eqnarray*}
Thus by Lemma \ref{CGTY}, there exists an element $\nu'$ in $\Gamma^{(2)}(P_{7})$ such that ${\rm lk}\left(\psi_{N_{9},P_{7}}(f)(\nu')\right)$ is odd. Hence by Proposition \ref{map2}, there exists an element $\mu'$ in $\Psi_{P_{7},N_{9}}^{(2)}\left(\Gamma^{(2)}(P_{7})\right)$ such that ${\rm lk}(f(\mu'))$ is odd. Note that $\Psi_{P_{7},N_{9}}^{(2)}\left(\Gamma^{(2)}(P_{7})\right)$ consists of nine elements 
\begin{eqnarray*}
&&\mu'_{1}= [3\ 4\ 5]\cup [1\ 2\ 8\ 7],\ 
\mu'_{2}= [1\ 5\ 4\ 7]\cup [2\ 3\ 9\ 8],\ 
\mu'_{3}= [2\ 8\ 5\ 4]\cup [3\ 1\ 7\ 9],\\
&&\mu'_{4}= [1\ 2\ 4\ 7]\cup [3\ 5\ 8\ 9],\ 
\mu'_{5}= [1\ 2\ 3]\cup [4\ 7\ 8\ 5],\ 
\mu'_{6}= [1\ 2\ 8\ 5]\cup [3\ 4\ 7\ 9],\\
&&\mu'_{7}= [2\ 3\ 4]\cup [1\ 5\ 8\ 7],\ 
\mu'_{8}= [7\ 8\ 9]\cup [1\ 2\ 4\ 5],\ 
\mu'_{9}= [1\ 5\ 3]\cup [2\ 8\ 7\ 4]. 
\end{eqnarray*}
For $i=1,2,\ldots,9$, let $J^{i}$ be the subgraph of $N_{9}$ which is $\mu\cup \mu'_{i}\cup \overline{6\ 9}$ if $i=3,6$ and $\mu\cup \mu'_{i}$ if $i\neq 3,6$. Assume that ${\rm lk}(f(\mu'_{i}))$ is odd for some $i\neq 8$. Then it can be easily seen that $J^{i}$ contains a graph $D^{i}$ as a minor so that $D^{i}$ is isomorphic to $D_{4}$ and $\left\{\mu,\mu'_{i}\right\}=\Psi_{D^{i},J^{i}}^{(2)}\left(\Gamma^{(2)}(D^{i})\right)$. Since both ${\rm lk}(f(\mu))$ and ${\rm lk}(f(\mu'_{i}))$ are odd, by Lemma \ref{D4type} there exists an element $\gamma$ in $\Gamma(J^{i})$ such that $a_{2}(f(\gamma))$ is odd. Next assume that ${\rm lk}(f(\mu'_{8}))$ is odd. We denote two elements $[7\ 8\ 9]\cup [1\ 2\ 6\ 5]$ and $[7\ 8\ 9]\cup [4\ 2\ 6\ 5]$ in $\Gamma^{(2)}(J^{8})$ by $\mu'_{8,1}$ and $\mu'_{8,2}$, respectively. We denote the subgraph $\mu\cup \mu'_{8,j}$ of $J^{8}$ by $J^{8,j}\ (j=1,2)$. Then it can be easily seen that $J^{8,j}$ contains a graph $D^{8,j}$ as a minor so that $D^{8,j}$ is isomorphic to $D_{4}$ and $\left\{\mu,\mu'_{8,j}\right\}=\Psi_{D^{8,j},J^{8,j}}^{(2)}\left(\Gamma^{(2)}(D^{8,j})\right)\ (j=1,2)$. Note that $[1\ 2\ 4\ 5] = [1\ 2\ 6\ 5] + [4\ 2\ 6\ 5]$ in $H_{1}(J^{8};{\mathbb Z}_{2})$, where $H_{*}(\cdot;{\mathbb Z}_{2})$ denotes the homology group with ${\mathbb Z}_{2}$-coefficients. Then, by the homological property of the linking number, we have that 
\begin{eqnarray*}
1 \equiv {\rm lk}(f(\mu'_{8})) \equiv {\rm lk}(f(\mu'_{8,1}))+{\rm lk}(f(\mu'_{8,2})) \pmod{2}. 
\end{eqnarray*}
Thus we see that ${\rm lk}(f(\mu'_{8,1}))$ is odd or ${\rm lk}(f(\mu'_{8,2}))$ is odd. In either case, by Lemma \ref{D4type} there exists an element $\gamma$ in $\Gamma(J^{8,j})$ such that $a_{2}(f(\gamma))$ is odd. 

\vspace{0.2cm}
\noindent
{\bf Case 2.} $\mu=[1\ 2\ 3]\cup [4\ 5\ 6]$. 

\vspace{0.2cm}
Note that $N_{9}$ contains $P_{9}$ as the proper minor 
\begin{eqnarray*}
\left(\left(\left(\left(\left(N_{9}-\overline{1\ 2}\right)-\overline{2\ 3}\right)-\overline{3\ 1}\right)-\overline{4\ 5}\right)-\overline{5\ 6}\right)-\overline{6\ 4}. 
\end{eqnarray*}
Thus by Lemma \ref{CGTY}, there exists an element $\nu'$ in $\Gamma^{(2)}(P_{9})$ such that ${\rm lk}\left(\psi_{N_{9},P_{9}}(f)(\nu')\right)$ is odd. Hence by Proposition \ref{map2}, there exists an element $\mu'$ in $\Psi_{P_{9},N_{9}}^{(2)}\left(\Gamma^{(2)}(P_{9})\right)$ such that ${\rm lk}(f(\mu'))$ is odd. Note that $\Psi_{P_{9},N_{9}}^{(2)}\left(\Gamma^{(2)}(P_{9})\right)$ consists of seven elements, and by the symmetry of $N_{9}$, we may assume that $\mu'=[1\ 5\ 8\ 7]\cup [2\ 6\ 9\ 3\ 4]$ or $[7\ 8\ 9]\cup [1\ 5\ 3\ 4\ 2\ 6]$ without loss of generality. We denote the subgraph $\mu\cup \mu'$ of $N_{9}$ by $J$. Assume that $\mu'=[1\ 5\ 8\ 7]\cup [2\ 6\ 9\ 3\ 4]$. We denote two elements $[1\ 5\ 8\ 7]\cup [4\ 3\ 2]$ and $[1\ 5\ 8\ 7]\cup [6\ 9\ 3\ 2]$ in $\Gamma^{(2)}(J)$ by $\mu'_{1}$ and $\mu'_{2}$, respectively. We denote the subgraph $\mu\cup \mu'_{i}$ of $J$ by $J^{i}\ (i=1,2)$. Then $J^{i}$ contains a graph $D^{i}$ as a minor so that $D^{i}$ is isomorphic to $D_{4}$ and $\left\{\mu,\mu'_{i}\right\}=\Psi_{D^{i},J^{i}}^{(2)}\left(\Gamma^{(2)}(D^{i})\right)\ (i=1,2)$. Since $[2\ 6\ 9\ 3\ 4] = [4\ 3\ 2] + [6\ 9\ 3\ 2]$ in $H_{1}(J;{\mathbb Z}_{2})$, it follows that 
\begin{eqnarray*}
1 \equiv {\rm lk}(f(\mu')) \equiv {\rm lk}(f(\mu'_{1}))+{\rm lk}(f(\mu'_{2})) \pmod{2}. 
\end{eqnarray*}
This implies that ${\rm lk}(f(\mu'_{1}))$ is odd or ${\rm lk}(f(\mu'_{2}))$ is odd. In either case, by Lemma \ref{D4type} there exists an element $\gamma$ in $\Gamma(J^{i})$ such that $a_{2}(f(\gamma))$ is odd. Next assume that $\mu'=[7\ 8\ 9]\cup [1\ 5\ 3\ 4\ 2\ 6]$. We denote four elements $[7\ 8\ 9]\cup [3\ 4\ 5]$, $[7\ 8\ 9]\cup [4\ 5\ 6]$, $[7\ 8\ 9]\cup [1\ 5\ 6]$ and $[7\ 8\ 9]\cup [2\ 4\ 6]$ in $\Gamma^{(2)}(J)$ by $\mu'_{1},\mu'_{2},\mu'_{3}$ and $\mu'_{4}$, respectively. Since $[1\ 5\ 3\ 4\ 2\ 6]= [3\ 4\ 5] + [4\ 5\ 6] + [1\ 5\ 6] + [2\ 4\ 6]$ in $H_{1}(J;{\mathbb Z}_{2})$, it follows that 
\begin{eqnarray*}
1\equiv {\rm lk}(\mu') 
\equiv {\rm lk}(\mu'_{1})+ {\rm lk}(\mu'_{2})+ {\rm lk}(\mu'_{3})+ {\rm lk}(\mu'_{4}) \pmod{2}.
\end{eqnarray*}
This implies that ${\rm lk}(\mu'_{i})$ is odd for some $i=1,2,3$ or $4$. Moreover, by the symmetry of $J$, we may assume that ${\rm lk}(\mu'_{1})$ is odd or ${\rm lk}(\mu'_{2})$ is odd without loss of generality. Assume that ${\rm lk}(\mu'_{1})$ is odd. We denote the subgraph $\mu \cup \mu'_{1}\cup \overline{1\ 7}\cup \overline{6\ 9}$ of $N_{9}$ by $J^{1}$. Then $J^{1}$ contains a graph $D^{1}$ as a minor so that $D^{1}$ is isomorphic to $D_{4}$ and $\left\{\mu,\mu'_{1}\right\}=\Psi_{D^{1},J^{1}}^{(2)}\left(\Gamma^{(2)}(D^{1})\right)$. Since both ${\rm lk}(f(\mu))$ and ${\rm lk}(f(\mu'_{1}))$ are odd, by Lemma \ref{D4type} there exists an element $\gamma$ in $\Gamma(J^{1})$ such that $a_{2}(f(\gamma))$ is odd. Next assume that ${\rm lk}(\mu'_{2})$ is odd. We denote four elements $[7\ 8\ 9]\cup [1\ 2\ 6]$, $[7\ 8\ 9]\cup [1\ 2\ 3]$, $[7\ 8\ 9]\cup [2\ 3\ 4]$ and $[7\ 8\ 9]\cup [1\ 3\ 5]$ in $\Gamma^{(2)}(J)$ by $\mu'_{5},\mu'_{6},\mu'_{7}$ and $\mu'_{8}$, respectively. Since $[1\ 5\ 3\ 4\ 2\ 6]= [1\ 2\ 6] + [1\ 2\ 3] + [2\ 3\ 4] + [1\ 3\ 5]$ in $H_{1}(J;{\mathbb Z}_{2})$, we have that 
\begin{eqnarray*}
1\equiv {\rm lk}(\mu') 
\equiv {\rm lk}(\mu'_{5})+ {\rm lk}(\mu'_{6})+ {\rm lk}(\mu'_{7})+ {\rm lk}(\mu'_{8}) \pmod{2}.
\end{eqnarray*}
Thus we see that ${\rm lk}(\mu'_{i})$ is odd for some $i=5,6,7$ or $8$. Moreover, by the symmetry of $J$, we may assume that ${\rm lk}(\mu'_{5})$ is odd or ${\rm lk}(\mu'_{6})$ is odd without loss of generality. Assume that ${\rm lk}(\mu'_{5})$ is odd. We denote the subgraph $\mu \cup \mu'_{5}\cup \overline{4\ 7}\cup \overline{3\ 9}$ of $N_{9}$ by $J^{5}$. Then $J^{5}$ contains a graph $D^{5}$ as a minor so that $D^{5}$ is isomorphic to $D_{4}$ and $\left\{\mu,\mu'_{5}\right\}=\Psi_{D^{5},J^{5}}^{(2)}\left(\Gamma^{(2)}(D^{5})\right)$. Since both ${\rm lk}(f(\mu))$ and ${\rm lk}(f(\mu'_{5}))$ are odd, by Lemma \ref{D4type} there exists an element $\gamma$ in $\Gamma(J^{5})$ such that $a_{2}(f(\gamma))$ is odd. Finally, assume that ${\rm lk}(\mu'_{6})$ is odd. Let us consider the 3-component link $L=f([1\ 2\ 3]\cup [4\ 5\ 6]\cup [7\ 8\ 9])$. Since all $2$-component sublinks of $L$ are $f(\mu),f(\mu'_{2})$ and $f(\mu'_{6})$, each of the $2$-component sublinks of $L$ has an odd linking number. 

Next we show in the case of $G=N'_{10}$. Let $f$ be a spatial embedding of $N'_{10}$. Note that $N'_{10}$ contains $P_{7}$ as the proper minor 
\begin{eqnarray*}
\left(\left(\left(N'_{10}-\overline{7\ 8}\right)-\overline{8\ 9}\right)-\overline{9\ 7}\right)/\overline{4\ 7}/\overline{5\ 8}/\overline{6\ 9}.
\end{eqnarray*}
Thus by Lemma \ref{CGTY}, there exists an element $\nu$ in $\Gamma^{(2)}(P_{7})$ such that ${\rm lk}\left(\psi_{N'_{10},P_{7}}(f)(\nu)\right)$ is odd. Hence by Proposition \ref{map2}, there exists an element $\mu$ in $\Psi_{P_{7},N'_{10}}^{(2)}\left(\Gamma^{(2)}(P_{7})\right)$ such that ${\rm lk}(f(\mu))$ is odd. Note that $\Psi_{P_{7},N'_{10}}^{(2)}\left(\Gamma^{(2)}(P_{7})\right)$ consists of nine elements, and by the symmetry of $N'_{10}$, we may assume that $\mu=[1\ 7\ 4\ 5]\cup [2\ 10\ 3\ 9\ 6]$, $[2\ 4\ 5\ 8]\cup [1\ 10\ 3\ 9\ 6]$, $[3\ 10\ 8\ 5]\cup [1\ 6\ 2\ 4\ 7]$, $[3\ 4\ 5]\cup [1\ 10\ 2\ 6]$ or $[2\ 8\ 10]\cup [1\ 6\ 9\ 3\ 4\ 7]$ without loss of generality. 

\vspace{0.2cm}
\noindent
{\bf Case 1.} $\mu=[1\ 7\ 4\ 5]\cup [2\ 10\ 3\ 9\ 6]$. 

\vspace{0.2cm}
Note that $N'_{10}$ contains $P_{9}$ as the proper minor 
\begin{eqnarray*}
\left(\left(\left(\left(\left(N'_{10}-\overline{5\ 1}\right)-\overline{5\ 3}\right)-\overline{5\ 4}\right)-\overline{5\ 6}\right)-\overline{5\ 8}\right)-\overline{7\ 9}. 
\end{eqnarray*}
Thus by Lemma \ref{CGTY}, there exists an element $\nu'$ in $\Gamma^{(2)}(P_{9})$ such that ${\rm lk}\left(\psi_{N'_{10},P_{9}}(f)(\nu')\right)$ is odd. Hence by Proposition \ref{map2}, there exists an element $\mu'$ in $\Psi_{P_{9},N'_{10}}^{(2)}\left(\Gamma^{(2)}(P_{9})\right)$ such that ${\rm lk}(f(\mu'))$ is odd. Note that $\Psi_{P_{9},N'_{10}}^{(2)}\left(\Gamma^{(2)}(P_{9})\right)$ consists of seven elements 
\begin{eqnarray*}
&&\mu'_{1}= [3\ 10\ 8\ 9]\cup [1\ 6\ 2\ 4\ 7],\ 
\mu'_{2}= [1\ 7\ 8\ 10]\cup [2\ 4\ 3\ 9\ 6],\\ 
&& \mu'_{3}= [1\ 10\ 2\ 6]\cup [3\ 4\ 7\ 8\ 9],\ 
\mu'_{4}= [2\ 4\ 3\ 10]\cup [1\ 7\ 8\ 9\ 6],\\
&& \mu'_{5}= [2\ 4\ 7\ 8]\cup [1\ 10\ 3\ 9\ 6],\ 
\mu'_{6}= [2\ 8\ 9\ 6]\cup [1\ 10\ 3\ 4\ 7],\\
&&\mu'_{7}= [2\ 8\ 10]\cup [1\ 6\ 9\ 3\ 4\ 7]. 
\end{eqnarray*}
For $i=1,2,\ldots,7$, let $J^{i}$ be the subgraph of $N'_{10}$ which is $\mu\cup \mu'_{i}\cup \overline{5\ 8}$ if $i=1,6,7$ and $\mu\cup \mu'_{i}$ if $i= 2,3,4,5$. Assume that ${\rm lk}(f(\mu'_{i}))$ is odd for some $i$. Then $J^{i}$ contains a graph $D^{i}$ as a minor so that $D^{i}$ is isomorphic to $D_{4}$ and $\left\{\mu,\mu'_{i}\right\}=\Psi_{D^{i},J^{i}}^{(2)}\left(\Gamma^{(2)}(D^{i})\right)$. Since both ${\rm lk}(f(\mu))$ and ${\rm lk}(f(\mu'_{i}))$ are odd, by Lemma \ref{D4type} there exists an element $\gamma$ in $\Gamma(J^{i})$ such that $a_{2}(f(\gamma))$ is odd. 

\vspace{0.2cm}
\noindent
{\bf Case 2.} $\mu=[2\ 4\ 5\ 8]\cup [1\ 10\ 3\ 9\ 6]$. 

\vspace{0.2cm}
Note that $N'_{10}$ contains another $P_{9}$ as the proper minor 
\begin{eqnarray*}
\left(\left(\left(\left(\left(N'_{10}-\overline{8\ 2}\right)-\overline{8\ 5}\right)-\overline{8\ 7}\right)-\overline{8\ 9}\right)-\overline{8\ 10}\right)-\overline{3\ 4}. 
\end{eqnarray*}
Thus by Lemma \ref{CGTY}, there exists an element $\nu'$ in $\Gamma^{(2)}(P_{9})$ such that ${\rm lk}\left(\psi_{N'_{10},P_{9}}(f)(\nu')\right)$ is odd. Hence by Proposition \ref{map2}, there exists an element $\mu'$ in $\Psi_{P_{9},N'_{10}}^{(2)}\left(\Gamma^{(2)}(P_{9})\right)$ such that ${\rm lk}(f(\mu'))$ is odd. Note that $\Psi_{P_{9},N'_{10}}^{(2)}\left(\Gamma^{(2)}(P_{9})\right)$ consists of seven elements 
\begin{eqnarray*}
&&\mu'_{1}= [1\ 6\ 9\ 7]\cup [2\ 4\ 5\ 3\ 10],\ 
\mu'_{2}= [1\ 7\ 4\ 5]\cup [2\ 10\ 3\ 9\ 6],\\ 
&& \mu'_{3}= [3\ 5\ 6\ 9]\cup [1\ 10\ 2\ 4\ 7],\ 
\mu'_{4}= [1\ 5\ 3\ 10]\cup [2\ 4\ 7\ 9\ 6],\\
&& \mu'_{5}= [1\ 10\ 2\ 6]\cup [3\ 9\ 7\ 4\ 5],\ 
\mu'_{6}= [1\ 5\ 6]\cup [2\ 4\ 7\ 9\ 3\ 10],\\
&&\mu'_{7}= [2\ 4\ 5\ 6]\cup [1\ 10\ 3\ 9\ 7]. 
\end{eqnarray*}
For $i=1,2,\ldots,7$, let $J^{i}$ be the subgraph of $N'_{10}$ which is $\mu\cup \mu'_{i}\cup \overline{7\ 8}$ if $i=1,7$ and $\mu\cup \mu'_{i}$ if $i\neq 1,7$. Assume that ${\rm lk}(f(\mu'_{i}))$ is odd for some $i$. Then $J^{i}$ contains a graph $D^{i}$ as a minor so that $D^{i}$ is isomorphic to $D_{4}$ and $\left\{\mu,\mu'_{i}\right\}=\Psi_{D^{i},J^{i}}^{(2)}\left(\Gamma^{(2)}(D^{i})\right)$. Since both ${\rm lk}(f(\mu))$ and ${\rm lk}(f(\mu'_{i}))$ are odd, by Lemma \ref{D4type} there exists an element $\gamma$ in $\Gamma(J^{i})$ such that $a_{2}(f(\gamma))$ is odd. 

\vspace{0.2cm}
\noindent
{\bf Case 3.} $\mu=[3\ 10\ 8\ 5]\cup [1\ 6\ 2\ 4\ 7]$. 

\vspace{0.2cm}
Let $P_{9}$ be the proper minor of $N'_{10}$ and $\mu'_{i}$ the element in $\Psi_{P_{9},N'_{10}}^{(2)}\left(\Gamma^{(2)}(P_{9})\right)\ (i=1,2,\ldots,7)$ as in Case 2. For $i=1,2,\ldots,7$, let $J^{i}$ be the subgraph of $N'_{10}$ which is $\mu\cup \mu'_{i}\cup \overline{8\ 9}$ if $i=1,4$ and $\mu\cup \mu'_{i}$ if $i\neq 1,4$. Assume that ${\rm lk}(f(\mu'_{i}))$ is odd for some $i$. Then $J^{i}$ contains a graph $D^{i}$ as a minor so that $D^{i}$ is isomorphic to $D_{4}$ and $\left\{\mu,\mu'_{i}\right\}=\Psi_{D^{i},J^{i}}^{(2)}\left(\Gamma^{(2)}(D^{i})\right)$. Since both ${\rm lk}(f(\mu))$ and ${\rm lk}(f(\mu'_{i}))$ are odd, by Lemma \ref{D4type} there exists an element $\gamma$ in $\Gamma(J^{i})$ such that $a_{2}(f(\gamma))$ is odd. 

\vspace{0.2cm}
\noindent
{\bf Case 4.} $\mu=[3\ 4\ 5]\cup [1\ 10\ 2\ 6]$. 

\vspace{0.2cm}
Note that $N'_{10}$ contains another $P_{7}$ as the proper minor 
\begin{eqnarray*}
\left(\left(\left(N'_{10}-\overline{3\ 4}\right)-\overline{4\ 5}\right)-\overline{5\ 3}\right)/\overline{3\ 9}/\overline{4\ 7}/\overline{5\ 8}. 
\end{eqnarray*}
Thus by Lemma \ref{CGTY}, there exists an element $\nu'$ in $\Gamma^{(2)}(P_{7})$ such that ${\rm lk}\left(\psi_{N'_{10},P_{7}}(f)(\nu')\right)$ is odd. Hence by Proposition \ref{map2}, there exists an element $\mu'$ in $\Psi_{P_{7},N'_{10}}^{(2)}\left(\Gamma^{(2)}(P_{7})\right)$ such that ${\rm lk}(f(\mu'))$ is odd. Note that $\Psi_{P_{7},N'_{10}}^{(2)}\left(\Gamma^{(2)}(P_{7})\right)$ consists of nine elements 
\begin{eqnarray*}
&&\mu'_{1}= [5\ 6\ 9\ 8]\cup [1\ 10\ 2\ 4\ 7],\ 
\mu'_{2}= [3\ 10\ 8\ 9]\cup [1\ 6\ 2\ 4\ 7],\\ 
&& \mu'_{3}= [1\ 5\ 8\ 10]\cup [2\ 4\ 7\ 9\ 6],\ 
\mu'_{4}= [7\ 8\ 9]\cup [1\ 10\ 2\ 6],\\
&& \mu'_{5}= [2\ 8\ 10]\cup [1\ 6\ 9\ 7],\ 
\mu'_{6}= [2\ 8\ 5\ 6]\cup [1\ 10\ 3\ 9\ 7],\\
&&\mu'_{7}= [1\ 7\ 8\ 5]\cup [2\ 10\ 3\ 9\ 6],\ 
\mu'_{8}= [1\ 5\ 6]\cup [2\ 4\ 7\ 9\ 3\ 10],\\
&& \mu'_{9}= [2\ 4\ 7\ 8]\cup [1\ 10\ 3\ 9\ 6]. 
\end{eqnarray*}
For $i=1,2,\ldots,9$, let $J^{i}$ be the subgraph of $N'_{10}$ which is $\mu\cup \mu'_{5}\cup \overline{4\ 7}\cup \overline{5\ 8}$ if $i=5$ and $\mu\cup \mu'_{i}$ if $i\neq 5$. Assume that ${\rm lk}(f(\mu'_{i}))$ is odd for some $i\neq 4,8$. Then $J^{i}$ contains a graph $D^{i}$ as a minor so that $D^{i}$ is isomorphic to $D_{4}$ and $\left\{\mu,\mu'_{i}\right\}=\Psi_{D^{i},J^{i}}^{(2)}\left(\Gamma^{(2)}(D^{i})\right)$. Since both ${\rm lk}(f(\mu))$ and ${\rm lk}(f(\mu'_{i}))$ are odd, by Lemma \ref{D4type} there exists an element $\gamma$ in $\Gamma(J^{i})$ such that $a_{2}(f(\gamma))$ is odd. Next assume that ${\rm lk}(f(\mu'_{8}))$ is odd. We denote two elements $[1\ 5\ 6]\cup [2\ 4\ 3\ 10]$ and $[1\ 5\ 6]\cup [3\ 4\ 7\ 9]$ in $\Gamma^{(2)}(J^{8})$ by $\mu'_{8,1}$ and $\mu'_{8,2}$, respectively. We denote the subgraph $\mu\cup \mu'_{8,1}$ of $J^{8}$ by $J^{8,1}$ and the subgraph $\mu\cup \mu'_{8,2}\cup \overline{8\ 9}\cup \overline{8\ 10}$ of $N'_{10}$ by $J^{8,2}$. Then $J^{8,j}$ contains a graph $D^{8,j}$ as a minor so that $D^{8,j}$ is isomorphic to $D_{4}$ and $\left\{\mu,\mu'_{8,j}\right\}=\Psi_{D^{8,j},J^{8,j}}^{(2)}\left(\Gamma^{(2)}(D^{8,j})\right)\ (j=1,2)$. Since $[2\ 4\ 7\ 9\ 3\ 10] = [2\ 4\ 3\ 10] + [3\ 4\ 7\ 9]$ in $H_{1}(J^{8};{\mathbb Z}_{2})$, it follows that 
\begin{eqnarray*}
1 \equiv {\rm lk}(f(\mu'_{8})) \equiv {\rm lk}(f(\mu'_{8,1}))+{\rm lk}(f(\mu'_{8,2})) \pmod{2}. 
\end{eqnarray*}
This implies that ${\rm lk}(f(\mu'_{8,1}))$ is odd or ${\rm lk}(f(\mu'_{8,2}))$ is odd. In either case, by Lemma \ref{D4type} there exists an element $\gamma$ in $\Gamma(J^{8,j})$ such that $a_{2}(f(\gamma))$ is odd. Finally assume that ${\rm lk}(f(\mu'_{4}))$ is odd. Note that $N'_{10}$ contains another $P_{9}$ as the proper minor 
\begin{eqnarray*}
\left(\left(\left(\left(\left(N'_{10}-\overline{2\ 4}\right)-\overline{2\ 6}\right)-\overline{2\ 8}\right)-\overline{2\ 10}\right)-\overline{5\ 1}\right)-\overline{5\ 3}. 
\end{eqnarray*}
Thus by Lemma \ref{CGTY}, there exists an element $\nu''$ in $\Gamma^{(2)}(P_{9})$ such that ${\rm lk}\left(\psi_{N'_{10},P_{9}}(f)(\nu'')\right)$ is odd. Hence by Proposition \ref{map2}, there exists an element $\mu''$ in $\Psi_{P_{9},N'_{10}}^{(2)}\left(\Gamma^{(2)}(P_{9})\right)$ such that ${\rm lk}(f(\mu''))$ is odd. Note that $\Psi_{P_{9},N'_{10}}^{(2)}\left(\Gamma^{(2)}(P_{9})\right)$ consists of seven elements 
\begin{eqnarray*}
&&\mu''_{1}= [5\ 6\ 9\ 8]\cup [1\ 10\ 3\ 4\ 7],\ 
\mu''_{2}= [4\ 5\ 8\ 7]\cup [1\ 10\ 3\ 9\ 6],\\ 
&& \mu''_{3}= [1\ 7\ 8\ 10]\cup [3\ 4\ 5\ 6\ 9],\ 
\mu''_{4}= [3\ 10\ 8\ 9]\cup [1\ 7\ 4\ 5\ 6],\\
&& \mu''_{5}= [1\ 6\ 9\ 7]\cup [3\ 4\ 5\ 8\ 10],\ 
\mu''_{6}= [3\ 9\ 7\ 4]\cup [1\ 10\ 8\ 5\ 6],\\
&&\mu''_{7}= [7\ 8\ 9]\cup [1\ 10\ 3\ 4\ 5\ 6]. 
\end{eqnarray*}
For $j=1,2,\ldots,7$, let $J^{4,j}$ be the subgraph of $N'_{10}$ which is $\mu'_{4}\cup \mu''_{j}\cup \overline{2\ 4}$ if $j=2,6$ and $\mu'_{4}\cup \mu''_{j}$ if $j\neq 2,6$. Assume that ${\rm lk}(f(\mu''_{j}))$ is odd for some $j\neq 7$. Then $J^{4,j}$ contains a graph $D^{4,j}$ as a minor so that $D^{4,j}$ is isomorphic to $D_{4}$ and $\left\{\mu'_{4},\mu''_{i}\right\}=\Psi_{D^{4,j},J^{4,j}}^{(2)}\left(\Gamma^{(2)}(D^{4,j})\right)$. Since both ${\rm lk}(f(\mu'_{4}))$ and ${\rm lk}(f(\mu''_{j}))$ are odd, by Lemma \ref{D4type} there exists an element $\gamma$ in $\Gamma(J^{4,j})$ such that $a_{2}(f(\gamma))$ is odd. Next assume that ${\rm lk}(f(\mu''_{7}))$ is odd. We denote three elements $[7\ 8\ 9]\cup [1\ 5\ 3\ 10]$, $[7\ 8\ 9]\cup [1\ 5\ 6]$ and $[7\ 8\ 9]\cup [3\ 4\ 5]$ in $\Gamma^{(2)}(N'_{10})$ by $\mu''_{7,1}$, $\mu''_{7,2}$ and $\mu''_{7,3}$, respectively. We denote the subgraph $\mu\cup \mu''_{7,k}\cup \overline{4\ 7}\cup \overline{2\ 8}$ of $N'_{10}$ by $J^{4,7,k}\ (k=1,2)$. Then $J^{4,7,k}$ contains a graph $D^{4,7,k}$ as a minor so that $D^{4,7,k}$ is isomorphic to $D_{4}$ and $\left\{\mu,\mu''_{7,k}\right\}=\Psi_{D^{4,7,k},J^{4,7,k}}^{(2)}\left(\Gamma^{(2)}(D^{4,7,k})\right)\ (k=1,2)$. Since $[1\ 10\ 3\ 4\ 5\ 6] = [1\ 5\ 3\ 10] + [1\ 5\ 6] + [3\ 4\ 5]$ in $H_{1}(N'_{10};{\mathbb Z}_{2})$, it follows that 
\begin{eqnarray*}
1 \equiv {\rm lk}(f(\mu'_{7})) \equiv {\rm lk}(f(\mu''_{7,1}))+{\rm lk}(f(\mu''_{7,2}))+ {\rm lk}(f(\mu''_{7,3})) \pmod{2}. 
\end{eqnarray*}
This implies that ${\rm lk}(f(\mu''_{7,k}))$ is odd for some $k$. If ${\rm lk}(f(\mu''_{7,1}))$ is odd or ${\rm lk}(f(\mu''_{7,2}))$ is odd, then by Lemma \ref{D4type} there exists an element $\gamma$ in $\Gamma(J^{4,7,k})$ such that $a_{2}(f(\gamma))$ is odd. If ${\rm lk}(f(\mu''_{7,3}))$ is odd, let us consider the 3-component link $L=f([3\ 4\ 5]\cup [7\ 8\ 9]\cup [1\ 10\ 2\ 6])$. Since all $2$-component sublinks of $L$ are $f(\mu),f(\mu'_{4})$ and $f(\mu''_{7,3})$, each of the $2$-component sublinks of $L$ has an odd linking number. 

\vspace{0.2cm}
\noindent
{\bf Case 5.} $\mu=[2\ 8\ 10]\cup [1\ 6\ 9\ 3\ 4\ 7]$. 

\vspace{0.2cm}
We denote two elements $[2\ 8\ 10]\cup [1\ 6\ 9\ 7]$ and $[2\ 8\ 10]\cup [3\ 9\ 7\ 4]$ in $\Gamma^{(2)}(N'_{10})$ by $\mu_{1}$ and $\mu_{2}$, respectively. Since $[1\ 6\ 9\ 3\ 4\ 7] = [1\ 6\ 9\ 7] + [3\ 9\ 7\ 4]$ in $H_{1}(N'_{10};{\mathbb Z}_{2})$, it follows that 
\begin{eqnarray*}
1 \equiv {\rm lk}(f(\mu)) \equiv {\rm lk}(f(\mu_{1}))+{\rm lk}(f(\mu_{2})) \pmod{2}. 
\end{eqnarray*}
This implies that ${\rm lk}(f(\mu_{1}))$ is odd or ${\rm lk}(f(\mu_{2}))$ is odd. By the symmetry of $N'_{10}$, we may assume that ${\rm lk}(f(\mu_{1}))$ is odd. Note that $N'_{10}$ contains another $P_{7}$ as the proper minor 
\begin{eqnarray*}
\left(\left(\left(N'_{10}-\overline{2\ 8}\right)-\overline{8\ 10}\right)-\overline{10\ 2}\right)/\overline{2\ 6}/\overline{3\ 10}/\overline{5\ 8}. 
\end{eqnarray*}
Thus by Lemma \ref{CGTY}, there exists an element $\nu'$ in $\Gamma^{(2)}(P_{7})$ such that ${\rm lk}\left(\psi_{N'_{10},P_{7}}(f)(\nu')\right)$ is odd. Hence by Proposition \ref{map2}, there exists an element $\mu'$ in $\Psi_{P_{7},N'_{10}}^{(2)}\left(\Gamma^{(2)}(P_{7})\right)$ such that ${\rm lk}(f(\mu'))$ is odd. Note that $\Psi_{P_{7},N'_{10}}^{(2)}\left(\Gamma^{(2)}(P_{7})\right)$ consists of nine elements 
\begin{eqnarray*}
&&\mu'_{1}= [3\ 5\ 8\ 9]\cup [1\ 6\ 2\ 4\ 7],\ 
\mu'_{2}= [1\ 7\ 8\ 5]\cup [2\ 4\ 3\ 9\ 6],\\ 
&& \mu'_{3}= [1\ 5\ 6]\cup [3\ 9\ 7\ 4],\ 
\mu'_{4}= [3\ 4\ 5]\cup [1\ 6\ 9\ 7],\\
&& \mu'_{5}= [5\ 6\ 9\ 8]\cup [1\ 10\ 3\ 4\ 7],\ 
\mu'_{6}= [4\ 5\ 8\ 7]\cup [1\ 10\ 3\ 9\ 6],\\
&&\mu'_{7}= [1\ 5\ 3\ 10]\cup [2\ 4\ 7\ 9\ 6],\ 
\mu'_{8}= [2\ 4\ 5\ 6]\cup [1\ 10\ 3\ 9\ 7],\\
&& \mu'_{9}= [7\ 8\ 9]\cup [1\ 10\ 3\ 4\ 2\ 6]. 
\end{eqnarray*}
For $i=1,2,\ldots,9$, let $J^{i}$ be the subgraph of $N'_{10}$ which is $\mu_{1}\cup \mu'_{3}\cup \overline{3\ 10}\cup \overline{5\ 8}$ if $i=3$ and $\mu_{1}\cup \mu'_{i}$ if $i\neq 3$. Assume that ${\rm lk}(f(\mu'_{i}))$ is odd for some $i\neq 4,9$. Then $J^{i}$ contains a graph $D^{i}$ as a minor so that $D^{i}$ is isomorphic to $D_{4}$ and $\left\{\mu_{1},\mu'_{i}\right\}=\Psi_{D^{i},J^{i}}^{(2)}\left(\Gamma^{(2)}(D^{i})\right)$. Since both ${\rm lk}(f(\mu_{1}))$ and ${\rm lk}(f(\mu'_{i}))$ are odd, by Lemma \ref{D4type} there exists an element $\gamma$ in $\Gamma(J^{i})$ such that $a_{2}(f(\gamma))$ is odd. Next assume that ${\rm lk}(f(\mu'_{9}))$ is odd. We denote two elements $[7\ 8\ 9]\cup [1\ 6\ 2\ 10]$ and $[7\ 8\ 9]\cup [2\ 4\ 3\ 10]$ in $\Gamma^{(2)}(J^{9})$ by $\mu'_{9,1}$ and $\mu'_{9,2}$, respectively. We denote the subgraph $\mu_{1}\cup \mu'_{8,1}$ of $J^{9}$ by $J^{9,1}$ and the subgraph $\mu_{1}\cup \mu'_{9,2}\cup \overline{5\ 3}\cup \overline{5\ 1}$ of $N'_{10}$ by $J^{9,2}$. Then $J^{9,j}$ contains a graph $D^{9,j}$ as a minor so that $D^{9,j}$ is isomorphic to $D_{4}$ and $\left\{\mu_{1},\mu'_{9,j}\right\}=\Psi_{D^{9,j},J^{9,j}}^{(2)}\left(\Gamma^{(2)}(D^{9,j})\right)\ (j=1,2)$. Since $[1\ 10\ 3\ 4\ 2\ 6] = [1\ 6\ 2\ 10] + [2\ 4\ 3\ 10]$ in $H_{1}(J^{9};{\mathbb Z}_{2})$, it follows that 
\begin{eqnarray*}
1 \equiv {\rm lk}(f(\mu'_{9})) \equiv {\rm lk}(f(\mu'_{9,1}))+{\rm lk}(f(\mu'_{9,2})) \pmod{2}. 
\end{eqnarray*}
This implies that ${\rm lk}(f(\mu'_{9,1}))$ is odd or ${\rm lk}(f(\mu'_{9,2}))$ is odd. In either case, by Lemma \ref{D4type} there exists an element $\gamma$ in $\Gamma(J^{9,j})$ such that $a_{2}(f(\gamma))$ is odd. Finally assume that ${\rm lk}(f(\mu'_{4}))$ is odd. Note that $N'_{10}$ contains another $P_{9}$ as the proper minor 
\begin{eqnarray*}
\left(\left(\left(\left(\left(N'_{10}-\overline{6\ 1}\right)-\overline{6\ 2}\right)-\overline{6\ 5}\right)-\overline{6\ 9}\right)-\overline{8\ 7}\right)-\overline{8\ 10}. 
\end{eqnarray*}
Thus by Lemma \ref{CGTY}, there exists an element $\nu''$ in $\Gamma^{(2)}(P_{9})$ such that ${\rm lk}\left(\psi_{N'_{10},P_{9}}(f)(\nu'')\right)$ is odd. Hence by Proposition \ref{map2}, there exists an element $\mu''$ in $\Psi_{P_{9},N'_{10}}^{(2)}\left(\Gamma^{(2)}(P_{9})\right)$ such that ${\rm lk}(f(\mu''))$ is odd. Note that $\Psi_{P_{9},N'_{10}}^{(2)}\left(\Gamma^{(2)}(P_{9})\right)$ consists of seven elements 
\begin{eqnarray*}
&&\mu''_{1}= [3\ 5\ 8\ 9]\cup [1\ 10\ 2\ 4\ 7],\ 
\mu''_{2}= [3\ 9\ 7\ 4]\cup [1\ 5\ 8\ 2\ 10],\\ 
&& \mu''_{3}= [1\ 7\ 4\ 5]\cup [2\ 8\ 9\ 3\ 10],\ 
\mu''_{4}= [2\ 4\ 5\ 8]\cup [1\ 10\ 3\ 9\ 7],\\
&& \mu''_{5}= [2\ 4\ 3\ 10]\cup [1\ 5\ 8\ 9\ 7],\ 
\mu''_{6}= [1\ 5\ 3\ 10]\cup [2\ 4\ 7\ 9\ 8],\\
&&\mu''_{7}= [3\ 4\ 5]\cup [1\ 10\ 2\ 8\ 9\ 7]. 
\end{eqnarray*}
For $j=1,2,\ldots,7$, let $J^{4,j}$ be the subgraph of $N'_{10}$ which is $\mu'_{4}\cup \mu''_{j}\cup \overline{2\ 6}$ if $j=4,5$ and $\mu'_{4}\cup \mu''_{j}$ if $j\neq 4,5$. Assume that ${\rm lk}(f(\mu''_{j}))$ is odd for some $j\neq 7$. Then $J^{4,j}$ contains a graph $D^{4,j}$ as a minor so that $D^{4,j}$ is isomorphic to $D_{4}$ and $\left\{\mu'_{4},\mu''_{i}\right\}=\Psi_{D^{4,j},J^{4,j}}^{(2)}\left(\Gamma^{(2)}(D^{4,j})\right)$. Since both ${\rm lk}(f(\mu'_{4}))$ and ${\rm lk}(f(\mu''_{j}))$ are odd, by Lemma \ref{D4type} there exists an element $\gamma$ in $\Gamma(J^{4,j})$ such that $a_{2}(f(\gamma))$ is odd. Next assume that ${\rm lk}(f(\mu''_{7}))$ is odd. We denote two elements $[3\ 4\ 5]\cup [1\ 10\ 8\ 9\ 7]$ and $[3\ 4\ 5]\cup [2\ 8\ 10]$ in $\Gamma^{(2)}(N'_{10})$ by $\mu''_{7,1}$ and $\mu''_{7,2}$, respectively. We denote the subgraph $\mu_{1}\cup \mu''_{7,1}\cup \overline{2\ 4}\cup \overline{5\ 6}$ of $N'_{10}$ by $J^{4,7}$. Then $J^{4,7}$ contains a graph $D^{4,7}$ as a minor so that $D^{4,7}$ is isomorphic to $D_{4}$ and $\left\{\mu_{1},\mu''_{7,1}\right\}=\Psi_{D^{4,7},J^{4,7}}^{(2)}\left(\Gamma^{(2)}(D^{4,7})\right)$. Since $[1\ 10\ 2\ 8\ 9\ 7] = [1\ 10\ 8\ 9\ 7] + [2\ 8\ 10]$ in $H_{1}(N'_{10};{\mathbb Z}_{2})$, it follows that 
\begin{eqnarray*}
1 \equiv {\rm lk}(f(\mu'_{7})) \equiv {\rm lk}(f(\mu''_{7,1}))+{\rm lk}(f(\mu''_{7,2})) \pmod{2}. 
\end{eqnarray*}
This implies that ${\rm lk}(f(\mu''_{7,1}))$ is odd or ${\rm lk}(f(\mu''_{7,2}))$ is odd. If ${\rm lk}(f(\mu''_{7,1}))$ is odd, then by Lemma \ref{D4type} there exists an element $\gamma$ in $\Gamma(J^{4,7})$ such that $a_{2}(f(\gamma))$ is odd. If ${\rm lk}(f(\mu''_{7,2}))$ is odd, let us consider the 3-component link $L=f([3\ 4\ 5]\cup [2\ 8\ 10]\cup [1\ 6\ 9\ 7])$. Since all $2$-component sublinks of $L$ are $f(\mu_{1}),f(\mu'_{4})$ and $f(\mu''_{7,2})$, each of the $2$-component sublinks of $L$ has an odd linking number. This completes the proof. 
\end{proof}

\begin{proof}[Proof of Theorem \ref{main2}.] 
Note that a graph in the Heawood family is obtained from one of $K_{7}$, $N_{9}$ and $N'_{10}$ by a finite sequence of $\triangle Y$-exchanges. Thus by Lemma \ref{DY}, Theorem \ref{N9FN} and the fact that $K_{7}$ is IK, namely I(K or C3L), it follows that every graph in the Heawood family is I(K or C3L). On the other hand, a graph in the Heawood family is obtained from one of $H_{12}$ and $C_{14}$ by a finite sequence of $Y \triangle$-exchanges. Since each of $H_{12}$ and $C_{14}$ is a minor-minimal IK graph and $\Gamma^{(3)}(H_{12})$ and $\Gamma^{(3)}(C_{14})$ are the empty sets, it follows that $H_{12}$ and $C_{14}$ are minor-minimal I(K or C3L) graphs. Thus by Lemma \ref{minor-min-lem}, we have the desired conslusion. 
\end{proof}

\begin{Remark}\label{apex}
A graph is said to be {\it $2$-apex} if it can be embedded in the $2$-sphere after the deletion of at most two vertices and all of their incidental edges. It is not hard to see that any $2$-apex graph may have a spatial embedding whose image contains neither a nontrivial knot nor a $3$-component link any of whose $2$-component sublink is nonsplittable. Thus any $2$-apex graph is not I(K or C3L). It is known that every graph of at most twenty edges is 2-apex \cite{M09} (see also \cite{JKM10}). Since the number of all edges of every graph in the Heawood family is twenty one, we see that any proper minor of a graph in the Heawood family is $2$-apex, namely not I(K or C3L). This also implies that any graph in the Heawood family is minor-minimal for I(K or C3L). 
\end{Remark}

\begin{Example}\label{not3L} 
Let $g_{9}$ be the spatial embedding of $N_{9}$ and $g'_{10}$ the spatial embedding of $N'_{10}$ as illustrated in Fig. \ref{no_3links}. Then it can be checked directly that both $g_{9}(N_{9})$ and $g'_{10}(N'_{10})$ do not contain a nonsplittable $3$-component link. Thus neither $N_{9}$ nor $N'_{10}$ is I3L. Moreover, we can see that $N_{10}, N_{11}, N'_{11}$ and $N'_{12}$ are not I3L in a similar way as the proof of Lemma \ref{main1lem}, see Fig. \ref{no_3links}. 
\end{Example}

\begin{figure}[htbp]
      \begin{center}
\scalebox{0.45}{\includegraphics*{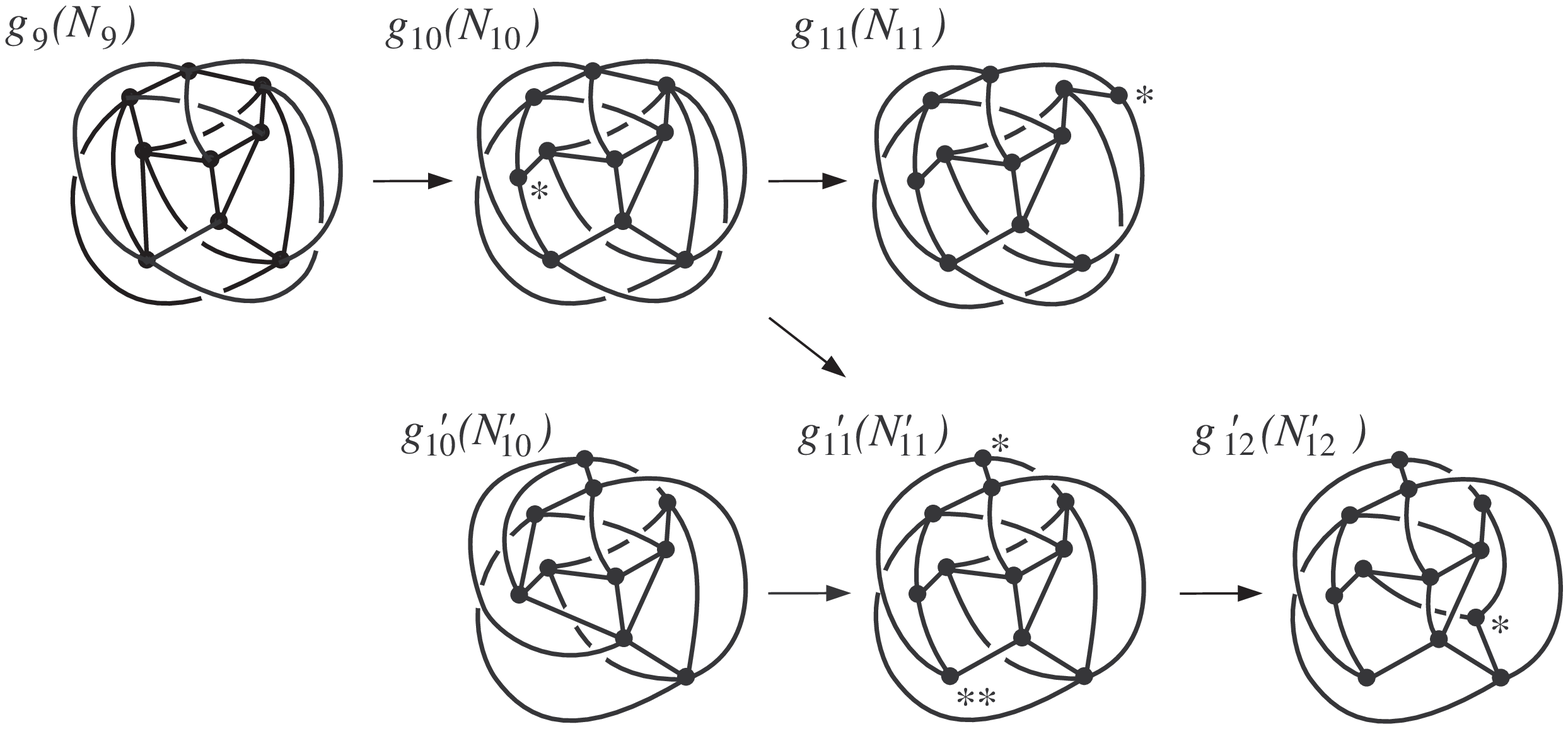}}
      \end{center}
   \caption{}
  \label{no_3links}
\end{figure} 

\begin{Remark}\label{Hea}
We remark that the Heawood graph is IK. The Heawood graph is the dual graph of $K_{7}$ which is embedded in a torus. It is known that there exists a unique graph $C_{14}$ obtained from $K_{7}$ by seven times applications of $\triangle Y$-exchanges \cite{KS92}. The seven triangles corresponds to the black triangles of black-and-white coloring of the torus by $K_{7}$. Then $C_{14}$ and $H$ are mapped to each other by parallel transformation of the torus, see Fig. \ref{torus}. Thus they are isomorphic. Since $C_{14}$ is IK, we have the result. 
\end{Remark}

\begin{figure}[htbp]
      \begin{center}
\scalebox{0.275}{\includegraphics*{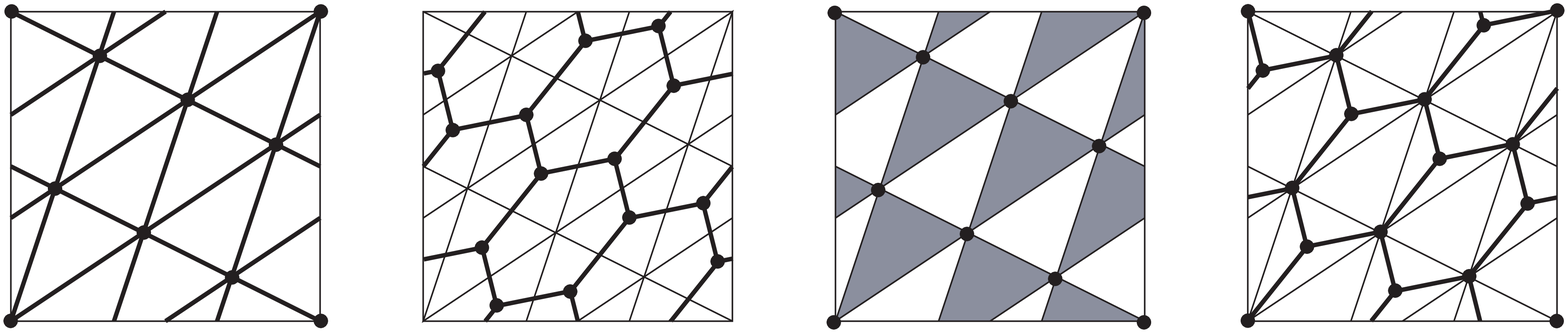}}
      \end{center}
   \caption{}
  \label{torus}
\end{figure} 

\begin{Remark}\label{K3311} 
{\rm 
It is known that all twenty six graphs obtained from the complete four-partite graph $K_{3,3,1,1}$ by a finite sequence of $\triangle Y$-exchanges are minor-minimal IK graphs \cite{KS92}, \cite{F02}. There exist thirty two graphs which are obtained from $K_{3,3,1,1}$ by a finite sequence of $\triangle Y$ and $Y \triangle$-exchanges but cannot be obtained from $K_{3,3,1,1}$ by a finite sequence of $\triangle Y$-exchanges. Recently, Goldberg-Mattman-Naimi announces that these thirty two graphs are also minor-minimal IK graphs \cite{GMN10}. 
}
\end{Remark}


%
{\normalsize
}

\end{document}